\documentclass{amsart}
\usepackage{graphicx, amssymb, eucal, hyperref, color,enumitem}
\newtheorem{iThm}{Theorem}

\newtheorem{thm}{Theorem}[section]
\newtheorem{cor}[thm]{Corollary}
\newtheorem{lem}[thm]{Lemma}

\newtheorem{prop}[thm]{Proposition}
\theoremstyle{definition}
\newtheorem{defn}[thm]{Definition}
\newtheorem{q}[thm]{Question}

\theoremstyle{remark}

\newtheorem{claim}{Claim}[thm]
\numberwithin{equation}{section}

\renewcommand{\phi}{\varphi}
\newcommand{\norm}[1]{\left\Vert#1\right\Vert}
\newcommand{\abs}[1]{\left\vert#1\right\vert}

\newcommand{\defined}[1]{\emph{#1}}
\newcommand{\mc}[1]{\mathcal{#1}}

\newcommand{\mbb}[1]{\mathbb{#1}}
\newcommand{\op}[1]{\operatorname{#1}}
\newcommand{\tp}{\op{tp}}
\newcommand{\qftp}{\op{qftp}}
\newcommand{\ignore}[1]{}
\newcommand{\ce}{\mathbb C}
\newcommand{\en}{\mathbb N}

\newcommand{\cV}{\mathcal V} 
\newcommand{\cU}{\mathcal U} 
 
\newcommand{\bbF}{\mbb F} 
 
\newcommand{\bbC}{\mbb C} 
\newcommand{\bbN}{\mbb N} 
\newcommand{\cO}{\mathcal O}

\DeclareMathOperator{\dist}{dist}
\newcommand{\e}{\epsilon} 
\newcommand{\allexist}{$\forall\exists$}
\DeclareMathOperator{\eq}{eq}
\newcommand{\cstar}{$\mathrm{C}^*$}
\newcommand{\cst}{\mathrm{C}^*}
 
\DeclareMathOperator{\Clop}{CL} 

\newcommand{\CCantor}{C\left(2^{\bbN}\right)}
\begin{document}

\title{Quantifier elimination in   \cstar-algebras}%
\author[C. J. Eagle]{Christopher J. Eagle}
\address[C. J. Eagle]{Department of Mathematics and Statistics\\ 
University of Victoria \\
PO BOX 1700 STN CSC \\
Victoria, British Columbia, Canada\\ 
V8W 2Y2}
\email{eaglec@uvic.ca}
\urladdr{http://www.math.uvic.ca/~eaglec/}

\author[I. Farah]{Ilijas Farah}
\address[I. Farah]{Department of Mathematics and Statistics\\
York University\\
4700 Keele Street\\
Toronto, Ontario\\ Canada, M3J
1P3}

\email{ifarah@mathstat.yorku.ca}
\urladdr{http://www.math.yorku.ca/$\sim$ifarah}

\author[E. Kirchberg]{Eberhard Kirchberg}
\address[E. Kirchberg]{Humboldt Universit\"at zu Berlin \\ Institut f\"ur Mathematik \\ Unter den Linden 6 \\ D-10099 Berlin, Germany}
\email{kirchbrg@mathematik.hu-berlin.de}

\author[A. Vignati]{Alessandro Vignati}
\address[A. Vignati]{Department of Mathematics and Statistics\\
York University\\
4700 Keele Street\\
Toronto, Ontario\\ Canada, M3J
1P3\\
}
\email{ale.vignati@gmail.com}
\urladdr{http://www.automorph.net/avignati}

\thanks{IF was partially supported by NSERC. AV is supported by a Susan Mann Scholarship}

\date{\today}%
\begin{abstract}
The only unital
 \cstar-algebras that admit elimination of quantifiers in continuous logic in the language of unital \cstar-algebras
 are  $\mathbb{C}, \mathbb{C}^2$, $C($Cantor space$)$ 
 and  $M_2(\mathbb{C})$.  We also prove that the theory of \cstar-algebras does not have model companion and 
 show that the theory of $M_n(\cO_{n+1})$ is not  \allexist-axiomatizable for any $n\geq 2$.  
  \end{abstract}
 
\maketitle
\section*{Introduction}

One of the key steps in using model theory in applications is to understand the \emph{definable} objects in models of a particular theory.  It is often the case that the objects which can be defined without the use of quantifiers have particularly natural descriptions, while definitions involving quantifiers are more difficult to analyze.  Quantifier elimination, which is the property that every definable object can be defined without using quantifiers, is therefore a highly desirable feature for a theory to possess.

Quantifier elimination is a matter of the formal language used to study the structures of interest.  It is easy to see that any theory can be extended to a theory with quantifier elimination in an expanded language by simply adding a new symbol for every object definable in the original one.  While such an expansion yields quantifier elimination, it does so without simplifying the task of determining which objects are definable.  The usefulness of quantifier elimination results therefore depends on using a natural language for the structures at hand, so that it is possible to give a useful description of the objects that can be defined in quantifier-free way.  For this reason we consider \cstar-algebras as structures in the language for \cstar-algebras introduced in \cite{Farah2014a}.  This standard language for \cstar-algebras contains symbols for the natural operations in a \cstar-algebra; when we consider unital algebras we often add a symbol for the multiplicative identity to form the language of unital \cstar-algebras.  These languages are sufficiently expressive that many natural classes of \cstar-algebras are either axiomatizable, or at least defined by the omission of certain types (many examples of this kind are given in \cite[Theorem~2.5.1 and Theorem~5.7.3]{Muenster}).  Nevertheless, these languages are also sufficiently limited that quantifier-free formulas are quite simple, being continuous combinations of norms of $^*$-polynomials with complex coefficients. 
We identify $2$ with $\{0,1\}$ and the Cantor space with the product space $2^{\bbN}$.

\begin{iThm}
The theories of unital \cstar-algebras that admit quantifier elimination (in the language of unital \cstar-algebras) are exactly the complete theories of $\bbC$, $\bbC^2$,  $M_2(\bbC)$ and $C(2^\bbN)$.  
The theories of $\bbC$ and $C_0(2^{\bbN}\setminus \{0\})$ admit quantifier elimination  in the language of \cstar-algebras without a symbol for a unit, and 
no theory of a noncommutative \cstar-algebra  admits quantifier elimination in this language.  
\end{iThm}

\begin{proof} The unital case is Theorem~\ref{T0}. 
The claims about not necessarily unital algebras are established in Proposition~\ref{P.C.C-0} and  Theorem~\ref{C0}.
\end{proof} 

We also prove that the theory of \cstar-algebras does not have model companion (Theorem~\ref{T:MC}) 
 and give natural examples of 
\cstar-algebras whose theories are not \allexist-axiomatizable (Corollary~\ref{C2}).  

Section \ref{sec:Prelim} contains preliminaries and tests for quantifier elimination.  In this section we also completely answer the question of which finite-dimensional \cstar-algebras have quantifier elimination.
In Section \ref{sec:Noncommutative} we prove our main results, implying in particular that $M_2(\bbC)$ is the only noncommutative  \cstar-algebra whose theory admits elimination of quantifiers.
In Section~\ref{sec:Questions} we show that the theory of unital \cstar-algebras does not have a model companion, and also obtain results related to the \allexist-axiomatizablity of some classes of \cstar-algebras.

The word embedding has a model theoretical sense: an embedding is a unital injective $^*$-homomorphism. By $A^{\cU}$ we denote an ultrapower of $A$ associated with an ultrafilter~$\cU$. All ultrafilters are assumed to be nonprincipal ultrafilters on $\bbN$. 

\subsection*{Acknowledgments}
We are indebted to Isaac Goldbring for suggesting Theorem~\ref{T:MC}, 
 and an exchange that lead to Theorem~\ref{T.ssa}. We would also like to thank Bradd Hart for helpful remarks
 and to the anonymous referee for a very detailed and useful report. 

Research presented in the Appendix was supported by the Fields undergraduate summer research program in July and August 2014. This research gave the initial impetus to study that resulted in the present paper. 

\section{Quantifier elimination}\label{sec:Prelim}
In this section we recall the model-theoretic framework for studying \cstar-algebras, as well as tests for quantifier elimination.  The reader interested in a more complete discussion of the model theory of \cstar-algebras can consult \cite{Farah2014a} or \cite{Muenster}.  For more on quantifier elimination in metric structures in general, see \cite[Section 13]{BenYaacov2008a}.

\begin{defn}
The \defined{formulas} for \cstar-algebras are recursively defined as follows.  In each case, $\bar{x}$ denotes a finite tuple of variables (which will later be interpreted as elements of a \cstar-algebra).
\begin{enumerate}
\item{
If $P(\bar{x})$ is a $^*$-polynomial with complex coefficients, then $\norm{P(\bar{x})}$ is a formula.
}
\item{
If $\phi_1(\bar{x}), \ldots, \phi_n(\bar{x})$ are formulas and $f : \mathbb{R}^n \to \mathbb{R}$ is continuous, then $f(\phi_1(\bar{x}), \ldots, \phi_n(\bar{x}))$ is a formula.
}
\item{
If $\phi(\bar{x}, y)$ is a formula and $n \in \mathbb{N}^+$, then $\sup_{\norm{y} \leq n}\phi(\bar{x}, y)$ and $\inf_{\norm{y} \leq n}\phi(\bar{x}, y)$ are formulas.
}
\end{enumerate}
We think of $\sup_{\norm{y} \leq n}$ and $\inf_{\norm{y} \leq n}$ as replacements for the first-order quantifiers $\forall$ and $\exists$, respectively.  A formula constructed using only clauses (1) and (2) of the definition is therefore said to be \defined{quantifier-free}.
\end{defn}
The definition above is slightly different from the one in \cite{Farah2014a}.  In particular, we have replaced their \emph{domains of quantification} by requiring that our suprema and infima range over closed $n$-balls of finite radius, but the difference is clearly 
cosmetic. 

If $\phi(\bar{x})$ is a formula, $A$ is a \cstar-algebra, and $\bar{a}$ is a tuple of elements of $A$ of the same length as the tuple $\bar{x}$, there is a natural way to evaluate $\phi$ in $A$ with $\bar{x}$ replaced by $\bar{a}$; the result is a real number denoted $\phi^A(\bar{a})$.

\begin{defn}
Let $A$ be a \cstar-algebra, and $\bar{a} \in A^n$ be a tuple of elements from~$A$.  The \defined{type} of $\bar{a}$ in $A$, denoted $\tp^A(\bar{a})$, is defined to be the set of all formulas $\phi(\bar{x})$ such that $\phi^A(\bar{a}) = 0$.  Similarly, the \defined{quantifier-free type} of $\bar{a}$, denoted $\qftp^A(\bar{a})$, is the set of all quantifier-free formulas $\phi(\bar{x})$ such that $\phi^A(\bar{a}) = 0$.  If the algebra $A$ is clear from the context we omit it from the notation.
\end{defn}

A formula without free variables is a \defined{sentence}.  A \defined{theory} $T$ is a set of sentences, and a \cstar-algebra $A$ is a \defined{model} of $T$ (written $A \models T$) if every sentence in $T$ takes the value $0$ when interpreted in $A$.  The \defined{theory of $A$}, $\op{Th}(A)$, is the set of all sentences which take value $0$ when interpreted in $A$.  
If $\op{Th}(A)=\op{Th}(B)$ then we say that $A$ and $B$ are \emph{elementarily equivalent} and write $A\equiv B$. 

A formula $\phi$ is \emph{weakly stable} if  for every $\e>0$ there exists $\delta>0$ such that for every \cstar-algebra $A$ and every $a\in A$,  $\phi(a)<\delta$ implies that the distance from $a$ to the zero-set of $\phi$ in $A$ is $<\e$.
In the language of logic of metric structures, the zero-sets of weakly stable formulas are precisely the \emph{definable} sets (as defined in \cite[Definition~9.16]{BenYaacov2008a}). See \cite[Lemma~2.1]{Mitacs2012} and \cite[Lemma~3.2.4]{Muenster} for details. 
It is shown in \cite[Theorem~9.17]{BenYaacov2008a}  (see also \cite[Theorem~3.2.2]{Muenster} for the treatment of sets definable in a not necessarily complete theory) that every formula involving quantification over a definable set is equivalent to a standard formula. We will use without mention the fact that in a unital \cstar-algebra the sets of unitaries, self-adjoints, positive elements, and projections are definable (see \cite[Example~3.2.6]{Muenster}).

\begin{defn}
A theory $T$ has \defined{quantifier elimination} if for every formula $\phi(\bar{x})$ and every $\epsilon > 0$ there is a quantifier-free formula $\psi_\epsilon(\bar{x})$ such that whenever $A \models T$ and $\bar{a}\in A^n$, $n$ being the length of $\bar{x}$, is a tuple of elements of norm less than or equal to $1$, we have
\[\abs{\phi^A(\bar{a}) - \psi_\epsilon^A(\bar{a})} \leq \epsilon.\]

By a standard abuse of language, we say that a \cstar-algebra $A$ has quantifier elimination if $\op{Th}(A)$ does.
\end{defn}

Determining whether or not a theory has quantifier elimination directly from the definition is often difficult.  Fortunately, there are several tests for quantifier elimination that are more useful in practice.  We will make use of two such tests.  The first test, which we will use most often, is the following well-known strengthening of 
 \cite[Proposition 13.6]{BenYaacov2008a}.  We include a proof for the convenience of the reader.
 The \emph{density character} of a metric structure is the minimal cardinality of a dense subset of its underlying metric space, 
 and in particular it is $\leq \aleph_0$ if and only if the metric structure is separable. 

\begin{prop} \label{P.QE} 
Let $L$ be a language of metric structures and let $T$ be an $L$-theory. Then the following are equivalent:
\begin{enumerate}
\item $T$ has quantifier elimination;
\item if $A$ and $B$ are models of $T$ of density character  $\leq |L|$ 
then every embedding of a finitely generated substructure $F$ of $A$ into $B$ can be extended to an embedding of $A$ into an elementary extension of $B$. 
\end{enumerate}
\end{prop}

\begin{proof}In  \cite[Proposition 13.6]{BenYaacov2008a} (see also  \cite[pp. 84-91]{HensonIovino})
it was proved that the following version of (2) is equivalent to (1):
\begin{enumerate}
\item [(3)] if $A$ and $B$ are models of $T$ of density $\leq |L|$ then every embedding of a substructure $F$ of $A$ into $B$ can be extended to an embedding of $A$ into an elementary extension of $B$.
\end{enumerate}
Since (3) clearly implies (2), it only remains to prove that (2) implies (3).

Let  $A$, $B$,  $F$, and an embedding $\iota\colon F\to B$ be as in (3).
We may assume that $F\subseteq B$ and $\iota$ is the identity map. 
Consider the expansion $L'$ of $L$ obtained by adding a constant $c_a$ for every $a\in A\cup B$. 
Define an $L'$-theory $T'$ to be the union of the elementary diagram of $B$, 
$\{\phi(c_{\bar b})\mid \bar b\in B, \phi\text{ is a formula}, \phi(\bar b)=0\}$, 
 and the atomic diagram  of $A$, 
$\{\phi(c_{\bar a})\mid \bar a\in A, \phi\text{ is a quantifier-free formula}, \phi(\bar a)=0\}$, 
 (see \cite[2.3(a)]{Muenster}). Then $T'$ has a model if and only if there exists 
an elementary extension $C$ of $B$ and an embedding of $A$ into $C$ that extends $\iota$
(\cite[Theorem~2.3.4 and Theorem~2.3.5]{Muenster}). It therefore suffices to prove that 
$T'$ is consistent. Fix a finite  $T_0\subseteq T'$ and let $F_0$ be the substructure of $F$ 
that contains all $a\in F$ such that $c_a$ appears in $T_0$. Let $A_0$  ($B_0$, respectively), 
 be an elementary submodel of $A$ ($B$, respectively)  of density character $\leq |L|$ 
 that contain all $a$ such that $c_a$ appears in $T_0$. 
 Then (2) implies that $T_0$ is consistent. Since $T_0$ was arbitrary, the compactness theorem implies that
$T'$ is consistent and (3) holds.  
 \end{proof}

If we assume that $A$ and $B$ are separable, we note that the elementary extension of $B$ required for the statement (2) can be found inside a countably saturated model (see, for example, the proof of \cite[Proposition 13.17]{HensonIovino}). Specializing to \cstar-algebras we can therefore state a more appealing weakening of the assertion (2) of Proposition~\ref{P.QE}:
  \begin{enumerate}[label=($\star$)]
\item\label{property:WeakerThanQE} whenever $F$ is a finitely-generated  \cstar-algebra, $\iota\colon F\to A$ and $\kappa\colon F\to A^{\mathcal U}$ are embeddings then there is an  
 embedding $\phi\colon A\to A^{\mathcal U}$ that makes the diagram commute. 
\end{enumerate}

If the language of interest has a specified symbol for the unit then all algebras in play are assumed to be unital, as well as the embeddings. To avoid redundancy of notation, we will refer to the unital or the nonunital version of \ref{property:WeakerThanQE}. We emphasize that the `nonunital version of \ref{property:WeakerThanQE}' is applied to unital \cstar-algebras if we are considering the language without a specified symbol for the unit.  
 In Section \ref{sec:Noncommutative} we usually take \ref{property:WeakerThanQE} as an hypothesis, meaning that the results proved hold for algebras satisfying the nonunital version of \ref{property:WeakerThanQE} and for unital algebras satisfying the unital version of \ref{property:WeakerThanQE}.

The second quantifier elimination test we will use applies to a more restricted class of theories.  Recall that a theory is \emph{$\omega$-categorical} if it has a unique separable model (up to isomorphism).  The following result is well-known, but does not appear to have been explicitly stated in the literature, so we provide a brief proof.  This test was used in \cite[Theorem 5.26]{EagleVignati} to show that $\CCantor$ has quantifier elimination.

\begin{prop}\label{prop:OmegaCategoricalQE}
Let $T$ be an $\omega$-categorical theory.  The following are equivalent:
\begin{enumerate}
\item{
$T$ has quantifier elimination
}
\item{
the separable model of $T$ is \emph{near ultrahomogeneous}, in the sense that if $M \models T$ is separable and $\bar{a}, \bar{b} \in M^n$ have the same quantifier-free type, then for each $\epsilon > 0$ there is an automorphism $\Psi$ of $M$ such that the distance between $\Psi(\bar{a})$ and $\bar{b}$ is less than $\epsilon$.
}
\end{enumerate}
\end{prop}

Since $\bar{a}$ and $\bar{b}$ have the same quantifier-free type if and only if 
the map $a_i\mapsto b_i$ extends to an isomorphism between the metric structures generated by $\bar{a}$ and
$\bar{b}$, (2) is equivalent to the following relative of  Proposition~\ref{P.QE} (2).  
\begin{enumerate}
\item[(3)] If $F$ is a finitely-generated substructure of $T$ and $\bar a$ is a tuple generating it 
 then for every $\e>0$ and every embedding $\iota\colon F\to T$ 
there exists an automorphism $\Psi$ of $T$ such that the distance between $\iota(\bar{a})$ and $\Psi(\bar{a})$ is less
than $\e$. 
\end{enumerate}

\begin{proof}
The direction (1) implies (2) is \cite[Corollary 12.11]{BenYaacov2008a}, together with the fact that in a theory with quantifier elimination two tuples with the same quantifier-free type have the same type.

For (2) implies (1), it suffices to show that if we assume (2) then in every model of $T$ any two tuples with the same quantifier-free type have the same type.  So suppose $N \models T$ and we have tuples $\bar{a}, \bar{b}$ from $N$ such that $\qftp(\bar{a}) = \qftp(\bar{b})$.  By the continuous logic version of the Ryll-Nardzewski theorem (\cite[Theorem 12.10]{BenYaacov2008a}) both $\tp(\bar{a})$ and $\tp(\bar{b})$ are isolated, and hence there are $\bar{a}_0$ and $\bar{b}_0$ from $M$ such that $\tp(\bar{a})= \tp(\bar{a}_0)$ and $\tp(\bar{b}) = \tp(\bar{b}_0)$ (see \cite[Theorem 12.6]{BenYaacov2008a}).  In particular, $\qftp(\bar{a}_0) = \qftp(\bar{b}_0)$, so it follows from (2) that $\tp(\bar{a}_0) = \tp(\bar{b}_0)$, and hence $\tp(\bar{a}) = \tp(\bar{b})$.
\end{proof}

The two quantifier elimination tests above apply to any theory of metric structures.  We now record some general consequences of quantifier elimination more specifically for \cstar-algebras.  We will apply these results in the subsequent sections to show no noncommutative \cstar-algebra other than $M_2(\bbC)$ admits elimination of quantifiers.  
The first of these results, Lemma~\ref{L.SpectrumType},  is straightforward but very useful as it gives an analytic description of a quantifier-free type of a tuple of commuting normal elements. 
The \emph{joint spectrum} of commuting normal elements $a_1, \dots, a_n$, 
$j\sigma(\bar a)$,  is the set of all
$\bar \lambda\in \bbC^n$ such that $\{\lambda_1-a_1, \lambda_2-a_2, \dots, \lambda_n-a_n\}$
generates a proper ideal. 

\begin{lem}\label{L.SpectrumType}
In any \cstar-algebra, for two finite tuples of commuting normal elements $\bar a$ and $\bar b$ the following conditions are equivalent:
\begin{enumerate}
\item $\bar a$ and $\bar b$ have the same quantifier-free type
\item $j\sigma(\bar a)=j\sigma(\bar b)$
\item the  \cstar-algebras generated by $\bar a$ and $\bar b$ are isomorphic via an isomorphism that 
sends $\bar a$ to $\bar b$. 
\end{enumerate} 
Consequently, if a \cstar-algebra $A$ has quantifier elimination, then two finite tuples of commuting normal elements in $A$ have the same type if and only if they have the same joint spectrum.
\end{lem}

\begin{proof} Let $\bar a$ and $\bar b$ be as in the hypothesis.  
Statements (1) and (3) are obviously equivalent. 
By \cite[Proposition 5.25]{EagleVignati} the joint spectrum $j\sigma(\bar{a})$ is quantifier-free definable from $\bar{a}$, and hence if $\qftp(\bar{a}) = \qftp(\bar{b})$ then $j\sigma(\bar{a}) = j\sigma(\bar{b})$.

A \emph{character} of a \cstar-algebra is a unital $^*$-homomorphism
into $\bbC$. 
By the Gelfand--Naimark theorem every unital abelian \cstar-algebra is naturally isomorphic to $C(X)$ 
where $X$ is the space  of its characters 
with respect to the weak$^*$-topology. 
As the joint spectrum of $a_1,\dots, a_n$ is equal to the set of all 
$(f(a_1), \dots, f(a_n))$ where $f$ ranges over all characters of $\cst(\bar a)$, 
characters of $\cst(\bar a)$ are in one-to-one correspondence with 
the elements of $j\sigma(\bar a)$ and $\cst(\bar a)\cong C(j\sigma(\bar a))$. 
This proves that (2) implies (3). 
\end{proof}

By the Weyl-von Neumann theorem (see e.g., \cite[Corollary~II.4.2]{Dav:C*}) it is true that if $a$ and $b$ are self-adjoint elements of the Calkin algebra such that $\sigma(a) = \sigma(b)$, then $\tp(a) = \tp(b)$.  This is not true, however, for normal elements.  If $s$ is the unilateral shift in $\mathcal B(H)$ then its image $\pi(s)$ under the quotient map is a unitary with full spectrum and (because of the Fredholm index obstruction) it satisfies $\|\pi(s)-u^2\|\geq 1$ for all unitaries $u$.  
As pointed out in the introduction to \cite{PhWe:Calkin}, this failure of quantifier elimination is one of the reasons why it was difficult to construct an outer automorphism of the Calkin algebra.

\subsection{Finite-dimensional \cstar-algebras}\label{section:findim}
To conclude this section we treat the case of finite-dimensional \cstar-algebras.  We will need the fact if $A$ is a metric structure each of whose domains of quantification is compact, then the diagonal embedding of $A$ into its ultrapower is surjective.  This is because every ultrafilter limit converges in a compact metric space.  In particular, if $A$ is a finite-dimensional \cstar-algebra then the diagonal embedding of $A$ into its ultrapower is surjective, and so $A$ and $A^\mathcal U$ are isomorphic.  The Keisler-Shelah theorem (see \cite[Theorem 5.7]{BenYaacov2008a}) asserts that two structure are elementarily equivalent if and only if they have isomorphic ultrapowers (for some ultrafilter on a sufficiently large---and  possibly uncountable---index set).  It follows from these facts, together with the fact that compactness of domains of quantification is preserved by elementary equivalence, that a finite-dimensional \cstar-algebra is the unique model of its theory. This is the continuous logic analogue of the well-known fact in discrete model theory that any finite structure is the unique model of its theory.  For more on this, see \cite[\S 5]{BenYaacov2008a}, in particular the remark preceding Proposition 5.3.

We say that a projection $p$ is \emph{minimal} if there is no proper subprojection and \emph{abelian} if $pAp$ is abelian. We are interested in a strengthening of these two properties and we say that $p$ is \emph{scalar} if $pAp\cong \bbC$.  
The set of scalar projections in a \cstar-algebra is  definable. In fact a projection $p$ is scalar if and only if $\phi(p) = 0$, where
\[
\phi(p)=\sup_{\norm{a}\leq 1}\inf_{\lambda\in \ce, |\lambda|\leq 1}\norm{pap-\lambda p}. 
\] 
Although the above expression quantifies over the complex unit disc, it is possible to interpret the expression as a formula in our formal language, so we may treat $\phi$ as a formula; see \cite[Remark 3.4.3]{Muenster} for details.  It is not difficult to see that $\phi$ is $\{0, 1\}$-valued on projections, and therefore both the set of projections $p$ with $\phi(p) = 0$ and the set of projections $p$ with $\phi(p) \neq 0$ are definable, since the set of projections is definable (see \cite[Example~3.2.6]{Muenster}). In fact, if $pAp\not\cong\bbC$, we have that $pAp$ contains a $2$-dimensional vector space. In this case there is an element $a\in pAp$ with $\norm{a}=1$ such that $\norm{a-\lambda p}\geq 1$ whenever $\lambda\in \bbC$.

In particular, if $p\in A$ is scalar then $p$ is scalar also when seen in $A^{\mathcal U}$. Note that if~$A$ is finite-dimensional, scalar and minimal projections coincide.

\begin{thm}\label{T.Finite} 
For a finite-dimensional \cstar-algebra $A$ the following are equivalent. 
\begin{enumerate}
\item Every commutative subalgebra of $A$ is isomorphic to $\bbC$ or to $\bbC^2$. 
\item $A$ is  isomorphic to one of $\mathbb{C}, \mathbb{C}^2$, or $M_2(\mathbb{C})$.
\item $A$ has quantifier elimination in the language of unital \cstar-algebras. 
\item $A$ satisfies the unital version of \ref{property:WeakerThanQE}. 
 \end{enumerate} 
\end{thm}
\begin{proof}
The equivalence of (1) if and only if (2)  is an easy consequence of the fact that every finite-dimensional \cstar-algebra is isomorphic to a direct sum of full matrix algebras. 
Clauses (3) and (4) are equivalent for finite-dimensional algebras by Proposition \ref{P.QE}  and since a finite-dimensional \cstar-algebra is the only model of its theory.

We prove that (3) implies (1). If (1) fails, 
then there are two projections $p$ and $q$ in $A$ which are both minimal, are orthogonal, and are such that $q \neq 1-p$.  
If $A$ has quantifier elimination then every nontrivial projection has the same type as $p$, and in particular, is minimal.  This contradicts the fact that $q+p$ is a nontrivial nonminimal projection.

We now prove that $M_2(\mathbb{C})$ has quantifier elimination, using Proposition \ref{P.QE}.  Every ultrapower of $M_2(\bbC)$ is isomorphic to $M_2(\bbC)$. If $M$ and $N$ are isomorphic unital subalgebras of $M_2(\bbC)$, then by the equivalence of (2) and (1) and easy computation the isomorphism of $M$ and $N$ is implemented by a unitary in $M_2(\bbC)$.  Therefore the isomorphism extends to an automorphism of $M_2(\bbC)$, and this completes the proof. 

We omit the proofs that $\bbC$ and $\bbC^2$ have quantifier elimination, which are similar but easier (see also Lemma~\ref{L.SpectrumType}). 
 \end{proof}

\begin{prop} 
For a finite-dimensional \cstar-algebra $A$ the following are equivalent. 
\begin{enumerate}
\item $A$ is  isomorphic to $\mathbb{C}$. 
\item $A$ has quantifier elimination in the language of  \cstar-algebras without a symbol for a unit. 
\item $A$ satisfies the nonunital version of \ref{property:WeakerThanQE}. 
 \end{enumerate} 
\end{prop}

\begin{proof} 
(1)$\Rightarrow (3)$ follows from the fact that the only finitely-generated substructure of $\mathbb{C}$ is $\mathbb{C}$ itself and that every embedding of $\mathbb{C}$ into itself is necessarily unital. That (2) is equivalent to (3) is given by the fact that every finite-dimensional \cstar-algebra is the only model of its theory, therefore we are left to prove that (2) implies (1). For this, note that if $A$ is finite-dimensional and not isomorphic to $\bbC$ then 
it has a nontrivial projection. It therefore admits both unital and nonunital  embeddings
of $\bbC$, violating Proposition~\ref{P.QE} (2). 
\end{proof}

If we expand the language of \cstar-algebras to include a trace, then every matrix algebra (considered with its canonical trace) has quantifier elimination.  This follows from \cite[Proposition 13.6]{BenYaacov2008a} and the fact that unital embeddings of matrix algebras are trace preserving.  Making this change to the language does not affect questions of definability, since the trace is already definable in matrix algebras (see \cite[Lemma 3.5.3 and Theorem 3.5.5]{Muenster}).

\section{Noncommutative \cstar-algebras}\label{sec:Noncommutative}

Our goal is to prove that 
 no noncommutative infinite-dimensional 
\cstar-algebra admits quantifier elimination. 
The proof proceeds by showing that if $A$ is noncommutative and infinite-dimensional
and it admits quantifier elimination, then $A$ is purely infinite and simple (Proposition~\ref{P.Cuntz.2}).

\begin{lem} \label{L.MASA}
Assume $A$ is a \cstar-algebra with no scalar projections. 
Then it contains a positive contraction of full spectrum. 
\end{lem} 

\begin{proof} Passing to a subalgebra, we may assume $A$ is separable. 
Let $(X,d)$ be a locally compact metric space such that $C_0(X)$ is isomorphic to a masa of $A$. 
By the continuous functional calculus we need to find $f\in C_0(X)$ whose range is a nontrivial interval.
Since $A$ has no scalar projections, $X$ has no isolated points and is therefore uncountable. 

Let us first consider the case when $X$ has an uncountable connected component~$Y$. 
Choose a point $y\in Y$ and $r>0$ small enough to have $\sup_{z\in Y}d(z,y)\geq r$ and 
that $\{x\in X: d(x,y)<r\}$ is relatively compact. 
Define $g\colon [0,\infty)\to [0,1]$ by $g(t)=\frac{2t}{r}$ if $t\leq r/2$, $g(t)=-\frac{2t}{r}+2$ if $r/2<t\leq r$, and $g(t)=0$ elsewhere.  
Then $f\colon X\to [0,1]$ defined by $f(x)=g(d(x,y))$ is in $C_0(X)$ and its range is equal to~$[0,1]$. 

If there is no such $Y$ then every connected component of $X$ consists of a single point and therefore $X$ is 
zero-dimensional. Being locally compact and with no isolated points, $X$ has a clopen subset homeomorphic to 
the Cantor set. Since the Cantor set maps continuously  onto $[0,1]$, we can find $f$ as required. 
\end{proof}

\begin{prop} \label{P.abelian} 
If $A$ is noncommutative, infinite-dimensional, and satisfies \ref{property:WeakerThanQE} then $A$ does not have a scalar projection.
\end{prop} 

\begin{proof} 
We first observe that if $A$ satisfies the hypotheses of the theorem and has one scalar projection, then in fact every projection in $A$ is scalar.  Let $p \in A$ be a scalar projection, and let $q \in A$ be a projection that is not scalar.  Consider the embedding $\iota : \ce \to A$ given by $\iota(z) = zq$, and the embedding $\kappa : \ce \to A^{\mc{U}}$ given by $\kappa(z) = zp$.  Since the set of scalar projections is definable, 
$p$ remains a scalar projection in $A^{\cU}$, so $pA^{\mc{U}}p \cong \ce$.  There can be no embedding of $A$ into $A^{\mc{U}}$ as in the conclusion of \ref{property:WeakerThanQE}, because such an embedding would embed $qAq$ into $pA^{\mc{U}}p \cong \ce$, which is impossible as $qAq\not\cong \ce$

Now fix a scalar projection $p$.  Since $A$ is infinite-dimensional, so is 
  $B=(1-p)A(1-p)$ (since $1-p$ is a multiplier of $A$
this is a subalgebra of $A$ even if $A$ is nonunital). 
If $q$ is a nonzero projection in $B$ then 
$p+q$ is a non-scalar projection in $A$, contradicting the above. 

It will therefore suffice to prove 
that $B$ has a nonzero projection. Suppose otherwise. 
 Since being projectionless is axiomatizable (by the argument of  \cite[3.6(a)]{Muenster}), 
 $B^{\cU}$ is projectionless. Clearly $B^{\cU}$ is isomorphic to $(1-p)A^{\cU} (1-p)$.  

 By Lemma \ref{L.MASA} there is a positive element $a\in B$ with $\sigma(a)=[0,1]$. 
 Let $b=a+p$. Then $\sigma(b)=[0,1]$ since $ap=0$ and therefore   $F=\cst(b)$ is isomorphic to 
$C((0,1])$ and in turn $F\cong \cst(a)$. Let $\iota\colon F\to A$ send $b$ to $b$ and $\kappa\colon F\to A^{\cU}$ send $b$ to $a$. Let $\phi\colon A\to A^{\cU}$ the embedding extending 
$\kappa$ whose existence is assured by \ref{property:WeakerThanQE}. Then $\phi(p)$ 
is a projection $\leq\phi(a)$, a contradiction.
\end{proof}

\begin{defn}[{\cite{cuntz1978dimension}}]
For positive elements $a$ and $b$ in a \cstar-algebra $A$ we write 
 $a\precsim b$, and say that $a$ is \emph{Cuntz-subequivalent} to $b$,
  if there is a sequence $\{z_n\}_{n\in\bbN}$ such that \[\lim_n\|z_nb z_n^*-a\|=0.\] We write $a\sim b$ if $a\precsim b\precsim a$. Note that $\sim$ is an equivalence relation.
\end{defn}

\begin{lem} \label{L.Cuntz.1} For every $x$ in every \cstar-algebra one has  $x^*x\precsim xx^*$. 
Moreover, for every $n$ there exists $z_n$ with $\|z_n\|\leq n$ such that $\|x^*x-z_n^*xx^*z_n\|<1/n$. 
\end{lem} 

\begin{proof} For $n\in \bbN$ let $f_n\colon [0,1]\to [0,n]$ be defined by 
$f_n(t)=t^{-1/2}$ if $t\geq 1/n^2$ and $f_n(t)=n$ if $t<1/n$. 
Let $z_n=xf_n(x^*x)$. Clearly $\|z_n\|\leq n$. 
Also we have the following computation, which takes place in the commutative algebra $\cst(x^*x)$: 
\[
z_n^* xx^* z_n=f_n(x^*x) (x^*x)^2 f_n(x^*x)=g_n(x^*x),
\]
where $g_n(t)=t$ if $t\geq 1/n^2$ and $g_n(t)=t(1-tn)$ if $t<1/n^2$. Since $|t-g_n(t)|<1/n$ we have that 
$\|x^*x-z_n^*xx^* z_n\|<1/n$, as required. 
\end{proof}

  For $k\geq 0$ let us temporarily write $a\sim_k b$ if 
\begin{enumerate}
\item $a$ and $b$ are positive, 
\item for every $n$ there is $z_n$ such that $\|b-z_n^*az_n\|<2^k/n$ and $\|z_n\|\leq n^{2^k}$, and
\item for every $n$ there is $y_n$ such that $\|a-y_n^*by_n\|<2^k/n$ and $\|y_n\|\leq n^{2^k}$. 
\end{enumerate}
Note that  $a\sim_k b$ and $b\sim_k c$ implies $a\sim_{k+1} c$, and that $a\sim_k b$ implies $a$ is Cuntz-equivalent to $b$
(i.e. $a\precsim b$ and $b\precsim a$) for all $k$. 
Also, for all $k$, the relation $a\sim_k b$ is encoded in $\tp(a,b)$. 

Note that, by Lemma~\ref{L.Cuntz.1}, $xx^*\sim_0 x^*x$.

\begin{lem}\label{lem:Cuntz2}
If $A$ is noncommutative, infinite-dimensional, and satisfies \ref{property:WeakerThanQE}, then there exist $a, b \in A$ such that $a$ and $b$ are orthogonal positive contractions with full spectrum.  Moreover, for any two orthogonal positive contractions with full spectrum $a$ and $b$ in $A$ we have 
$a\sim_0 b$.
\end{lem}

\begin{proof}
Since $A$ is noncommutative, by \cite[II.6.4.14]{BlackadarOpAlg}, there is $x$ such that $\|x\|=1$ and $x^2=0$.
Then $xx^*$ and $x^*x$ are orthogonal positive elements of norm $1$. 
Since the spectra of $a=x^*x$ and $b=xx^*$ both contain $0$, they are equal (for all $x,y\in A$, $\sigma(xy)$ and $\sigma(yx)$ may only differ at $\{0\}$, see e.g., \cite[II.1.4.2]{BlackadarOpAlg}). 

Let us prove that we may assume $\sigma(a)=[0,1]$. 
If $\sigma(a)\neq [0,1]$ then by continuous functional calculus we can find a nonzero projection 
$p\in \cst(a)$. Since, by Proposition~\ref{P.abelian},  $A$ has no scalar projections, the algebra $pAp$ is infinite-dimensional 
and by Lemma~\ref{L.MASA} we can find positive $a_1\in pAp$ such that $\sigma(a_1)=[0,1]$. 
Let $x_1=xa_1$. Note that $a_1x\in pApx$ and $px=0$, hence $(x_1)^2=0$. 
Also, if we let $a_2 = x_1^*x_1$ then $a_2=a_1x^*xa_1=a_1pa_1=a_1^2$ and hence $\sigma(a_2)=[0,1]$. 
Therefore by replacing $x$ with $x_1$ and re-evaluating $a$ and $b$ 
we may assume $\sigma(a)=[0,1]$.  

If $c$ and $d$ are positive orthogonal elements with $\sigma(c)=\sigma(d)=[0,1]$ we have that $j\sigma(a,b)=j\sigma(c,d)=\{0\}\times [0,1]\cup [0,1]\times \{0\} $, so $\cst(a,b)\cong \cst(c,d)$; let $F = \cst(a,b)$. Let $\iota : F \to A$ be the inclusion map, let $\kappa\colon F\to A^{\mathcal U}$ be the embedding that sends $a$ to $c$ and $b$ to $d$, and let $\phi\colon A\to A^\cU$ be the embedding whose existence is guaranteed by \ref{property:WeakerThanQE}.  Let $x_n,y_n$ witness that $a\sim_0 b$.  Then $c\sim_0 d$ in $A^\cU$, as witnessed by $\phi(x_n)$ and $\phi(y_n)$. Since the diagonal embedding of $A$ into $A^{\mc{U}}$ is elementary, we have $\tp^A(c,d)=\tp^{A^{\cU}}(c,d)$, and in particular 
\[
\left(\inf_{\|y\|\leq n}( \|c-ydy^*\|)\right)^{A^\cU}<1/n \text{ and } \left(\inf_{\|y\|\leq n}( \|d-ycy^*\|)\right)^{A^\cU}<1/n
\]
 hence 
\[
\left(\inf_{\|y\|\leq n}( \|c-ydy^*\|)\right)^{A}<1/n \text{ and } \left(\inf_{\|y\|\leq n}( \|d-ycy^*\|)\right)^{A}<1/n.
\]

\end{proof}

A \cstar-algebra is \emph{purely infinite and simple} if it has dimension greater than 1 and for every two nonzero positive elements $a$ and $b$ we have $a\precsim b$
 (see \cite[\S 4.1]{rordam2002classification}).

\begin{prop} \label{P.Cuntz.2} Let $A$ be noncommutative and infinite-dimensional. If~$A$ satisfies \ref{property:WeakerThanQE} then $A$ is purely infinite and simple. 
\end{prop} 

\begin{proof}
Since being purely infinite and simple is elementary (see \cite[Theorem 2.5.1]{Muenster} or \cite{Goldbring2014}) it suffices to prove that $A^{\mathcal U}$ is purely infinite and simple. We will use the fact that for $a, b \in A$, we have $(a\precsim b)^A$ if and only if $(a\precsim b)^{A^\cU}$ (see \cite[Lemma 8.1.3]{Muenster}).

Before doing so we will need two preliminary claims.

\begin{claim}\label{c:Cuntz1}
Suppose that $f, g \in A^\mathcal U$ are positive contractions with full spectrum, and $fg=gf=g$.  Then $f\sim_1 g$.
\end{claim}

\begin{proof}
Choose elements $a$, $b$ and $c$ in $A$ such that $ab=0$ and $bc=cb=c$, each with full spectrum.  Such elements can be found in $C([0,1))$ and Lemma \ref{lem:Cuntz2} implies that 
$A$ contains a  copy of $C([0,1))$.  Then again by Lemma \ref{lem:Cuntz2} we have that $c\sim_0 a\sim_0 b$. Let $v_n,w_n\in A$ be witnessing that $a\sim_0 b$ and $y_n,z_n\in A$ be witnessing that $a\sim_0 c$.  Since the spectra of $f$ and $g$ are both $[0, 1]$ and $fg=gf=g$ we have $\qftp(f, g) = \qftp(b, c)$, hence $F=\cst(f,g)\cong \cst(b,c)$. Let $\iota \colon F \to A$ be the inclusion map and $\kappa \colon F \to A$ be the embedding sending $f$ to $b$ and $g$ to $c$.  Then by \ref{property:WeakerThanQE} there is $\phi\colon A\to A^{\mathcal U}$ an embedding making the diagram commute. It is easy to see that $\phi(v_n)$, $\phi(w_n)$ witness that $\phi(a)\sim_0 f$ and $\phi(y_n)$ and $\phi(z_n)$ witness that $\phi(a)\sim_0 g$, so $f\sim_1 g$.
\end{proof}

 \begin{claim}\label{c:Cuntz2} If $A$ is unital then 
every positive contraction $f\in A^{\mathcal U}$ with full spectrum satisfies $1 \precsim f$.
\end{claim}
\begin{proof}[Proof of Claim \ref{c:Cuntz2}]
Choose positive contractions $a, b, c$ in $A^{\mathcal U}$  each with full spectrum and such that $ab=bc=ac=0$, and let $d=1-a$. Then $\sigma(d)=[0,1]$ and $db=bd=b$, hence $d\sim_1 b$ by Claim \ref{c:Cuntz1}.  Also $a\sim_0 c$. In fact, we will only need to know that $d\precsim b$ and $a\precsim c$.  
 Since $bc=0$, by \cite[Proposition~1.1]{cuntz1978dimension} we have that $a+d\precsim  b+c$.  But $a+d=1$ and $\sigma(b+c)=[0,1]$.  In particular, $b+c$ is a positive contraction with full spectrum such that $1 \precsim b+c$. The same argument used in Claim \ref{c:Cuntz1} and in Lemma \ref{lem:Cuntz2} shows that $1\precsim f$ for every $f\in A^\cU$ as required for the claim.
\end{proof}

Now given positive contractions $a, b$ we show that $b \precsim a$. 
By Lemma~\ref{L.MASA} for every $0\leq r<1$ there exists a positive contraction $c$ in the hereditary\footnote{Recall that a subalgebra $B\subseteq A$ is hereditary if $c\leq d\in B$ implies $c\in B$ for all $c,d$ positive elements of $A$.}
subalgebra $\overline{(a-r)_+A^{\cU}(a-r)_+}$ such that $\sigma(c)=[0,1]$. By countable saturation of 
$A^{\cU}$ there is a positive contraction $d$ with full spectrum such that 
$da=d$, hence $d\precsim a$. It therefore suffices to prove $b\precsim d$, and by replacing $a$
with $d$ we may assume $\sigma(a)=[0,1]$. 

If $A^{\cU}$ is unital then by Claim~\ref{c:Cuntz2} we have $1\precsim a$ and (since $b\leq 1$) 
$b\precsim a$ follows. 

It remains to consider the case when  $A^{\cU}$ is nonunital.  
We prove that there exists a positive contraction $e$ such that $ae=a$ and $be=b$.  
Since $\lim_n \|aa^{1/n}-a\|=0$ and $a^{1/n}$ is a positive contraction  
the type of a positive contraction $c_1$ such that $ac_1=a$ is consistent and 
such contraction exists in $A^{\cU}$. For the same reason we have that there is $e$ with $(a+b)e=a+b=e(a+b)$. Let $B=\{x\mid xe=x=ex\}$. This is an hereditary \cstar-algebra of $A^{\cU}$, and since $a+b\in B$, then $a,b\in B$ as required.

Since $b\precsim e$ it suffices to prove $e\precsim a$.  
 If $\sigma(e)=[0,1]$
    this follows by  Claim~\ref{c:Cuntz1}. 
Otherwise, since $A^{\cU}$ is nonunital by countable saturation we can find a 
positive nonzero $f$ such that $fe=0$. By  Lemma~\ref{L.MASA}    we may assume $\sigma(f)=[0,1]$. 
Then $\sigma(e+f)=[0,1]$ and we can apply Claim~\ref{c:Cuntz1} to $e+f$ and complete the proof. 
\end{proof}

\subsection{$\cO_2$ and quantifier elimination}\label{S.O2}
Any  \cstar-algebra generated by $n$ isometries with orthogonal ranges with sum 1 is 
 isomorphic 
 to the Cuntz algebra $\mc{O}_n$ (\cite{Cuntz:Simple}).
 Hence $\cO_2$ is  the universal algebra defined by the relations $s^*s=t^*t=1$ and $ss^*+tt^*=1$. 
This algebra plays a pivotal role in Elliott's classification program (see \cite[Chapter 5]{rordam2002classification}).  
Notably,  $\mc{O}_2$ has some properties implied by quantifier elimination; for example, every unital embedding 
of $\cO_2$ into itself, or into any other model of its theory, is elementary (see e.g., \cite{Goldbring2014} or \cite[Proposition 2.15]{SSA}).  Nevertheless, we show below that $\mc{O}_2$ does not have quantifier elimination.

Our main goal in this section is to prove the following Theorem whose proof 
 extends ideas  used in the proof that stably finite exact \cstar-algebras 
are not necessarily embeddable into a stably finite nuclear \cstar-algebra  (see the discussion preceding 
\cite[Corollary~4.2.3]{Brown:Invariant}).

\begin{thm}\label{thm:wrongembeddings}
If $A$ is a separable, infinite-dimensional, noncommutative \cstar-algebra  then $A$ does not satisfy \ref{property:WeakerThanQE}. In particular, it does not have quantifier elimination.
\end{thm}

\begin{lem}\label{lem:O_2Embeds}
Let $A$ be an  infinite-dimensional noncommutative \cstar-algebra that satisfies \ref{property:WeakerThanQE}. 
  Then $\mc{O}_2$ embeds in $A$.
\end{lem}

\begin{proof} We write $p\sim q$ if $p$ and $q$ are Murray--von Neumann equivalent projections
and note that $p\sim q$ in $A$ if and only if $p\sim q$ in $A^{\cU}$. Using \ref{property:WeakerThanQE} we see that if $(p_1,q_1)$ and $(p_2,q_2)$ are commuting pairs of projections with $p_1\sim q_1$ and $j\sigma(p_1,q_1)=j\sigma(p_2,q_2)$, then $p_2\sim q_2$.
By Proposition~\ref{P.Cuntz.2} we know that $A$ is purely infinite and simple.
By \cite[Proposition~4.1.1 (iii)]{rordam2002classification} every projection $p$ in $A$ is
 \emph{properly infinite}, meaning that there are partial isometries $s$ and $s_1$ satisfying $s^*s=s_1^*s_1=p$ but $s_1s_1^*$ and $ss^*$
are orthogonal and such that $ss^*+s_1s_1^*\leq p$.  Therefore $A$ contains orthogonal Murray--von Neumann equivalent projections $ss^*$ and 
$s_1s_1^*$ and by the above $ss^*\sim p-ss^*$. By transitivity $p-ss^*\sim p$, 
and if  $t$ is such that $t^*t=p$ and $tt^*=p-ss^*$ then $s$ and $t$ are generators of a unital 
copy of $\cO_2$   in $pAp$. To prove the second assertion note that, if $A$ is unital $p$ can be chosen to be $1$.
\end{proof}

In the proof of Theorem~\ref{thm:wrongembeddings} we will make use of the reduced group \cstar-algebra $C_r^*(\mbb F_2)$ of the free group on two generators $\bbF_2$, constructed from the left regular representation $\lambda$ of $\bbF_2$ on $\ell^2(\bbF_2)$. For more information on the construction of the reduced \cstar-algebra of a group we refer to \cite[II.10.2.5]{BlackadarOpAlg}.  The algebra $C_r^*(\mbb F_2)$ is exact (see  \cite[p. 453, 1., 1-3]{Kirch93}, or  \cite[Proposition~5.1.8]{BrOz:C*}) and therefore embeds into $\mc{O}_2$ (see \cite{KircPhi:Embedding}).

\begin{proof}[Proof of Theorem~\ref{thm:wrongembeddings}] 
By Lemma \ref{lem:O_2Embeds}, $\cO_2$ embeds into $A$. If $A$ is unital, such an embedding can be chosen to be unital.
Let $\mc{U}$ be a nonprincipal ultrafilter on $\mathbb{N}$.  Each $M_n(\mathbb{C})$ embeds into $\cO_2$, and hence also embeds into $A$.  We therefore have an  embedding of the ultraproduct $M=\prod_{\mathcal U} M_n(\mathbb{C})$ inside $\cO_2^{\cU}\subseteq A^{\mathcal U}$, denoted by $\iota_2\colon M\to A^{\mathcal U}$. As before, if $A$ is unital, so is $\iota_2$.
Let $F=\cst_r(\bbF_2)\subseteq\mathcal O_2$. By \cite[Theorem~B]{HaagThor05}, $F$ is MF, that is, $F$ embeds into $M$ (see, for example, \cite[V.4.3.6]{BlackadarOpAlg}), and we can fix a unital embedding $\psi\colon F\to M$. Let $\kappa=\iota_2\circ \psi$. $\kappa$ is a $^*$-homomorphism from $F$ into $A^{\cU}$. 

We claim that $\kappa$ cannot be extended to an embedding $\phi$ of $\cO_2$ into $A^{\cU}$
(and in particular it cannot be extended to an embedding of $A$ into $A^{\cU}$), as \ref{property:WeakerThanQE} would require. 
Otherwise, by the nuclearity of $\cO_2$ and the Choi--Effros lifting theorem (\cite{choi1976completely}) 
there exists a  completely positive contraction
 $\psi'\colon \cO_2\to \ell_\infty(A)$ such that $\phi=\pi\circ\psi'$, where $\pi\colon \ell_\infty(A)\to A^{\cU}$ 
is the quotient map. 

Since each $M_n(\bbC)$ is an injective von Neumann algebra, by \cite[Proposition IV.2.1.4]{BlackadarOpAlg} there are coordinatewise conditional expectations $\theta_n\colon A\to M_n(\bbC)$.  Let $\theta\colon \ell_{\infty}(A)\to \prod M_n(\bbC)$
 be the conditional expectation induced by the 
 expectations~$\theta_n$. Then $\theta\circ\psi'$  a completely positive contraction, hence
 \[
 \theta\circ\psi'\colon F\to\prod M_n(\bbC)
 \] 
 is a completely positive contractive lifting for $\kappa$.  
We have therefore constructed an embedding of $\cst_r(\bbF_2)$ into $M$ with a completely positive contractive lifting. 
\cstar-algebras with this property are said to be quasidiagonal (see \cite{brown2000quasidiagonal}). 
However, by a result of Rosenberg (\cite[Corollary~7.1.18]{BrOz:C*}) 
quasidiagonality of $\cst_r(\mbb{F}_2)$ implies amenability of the nonamenable group $\mbb{F}_2$.  This contradiction concludes the proof. 
 \end{proof}

 \begin{thm}\label{T0}
The only \cstar-algebras with quantifier elimination in the language of unital \cstar-algebras are $\bbC$, $\bbC^2$, $M_2(\bbC)$ and $C(2^{\bbN})$. 
\end{thm} 

 \begin{proof}
That the given list includes all finite-dimensional examples is Theorem \ref{T.Finite}, and that the list includes all noncommutative examples is Theorem \ref{thm:wrongembeddings}. 
Every separable \cstar-algebra elementarily equivalent to $C(\beta\bbN\setminus \bbN)$ is isomorphic to $C(X)$ for a compact metrizable 0-dimensional space $X$ without isolated points, and therefore isomorphic to $C(2^{\bbN})$. 
In  \cite[Theorem 5.26]{EagleVignati}, 
it was proved that $C(\beta\bbN\setminus \bbN)$ (and therefore $C(2^\bbN)$)
admits 
quantifier elimination. In the appendix (written with D.C. Amador, B. Hart, J. Kawach, and S. Kim), we show that if there is a commutative example not on our list then it is of the form $C(X)$ for an indecomposable continuum~$X$. Finally, in \cite[Corollary 3.4]{EagleGoldVig} it was proved that no such commutative example exists.  
\end{proof} 

We now focus on the nonunital case. The noncommutative case follows almost directly from Theorem \ref{thm:wrongembeddings}:

\begin{thm}\label{C0}
Every noncommutative \cstar-algebra fails the nonunital variant of~\ref{property:WeakerThanQE}. 
Therefore no  noncommutative \cstar-algebra admits elimination of quantifiers in the language of \cstar-algebras without a symbol for a unit. 
\end{thm}

\begin{proof} Assume $A$ satisfies the nonunital version of 
\ref{property:WeakerThanQE}. 
By Theorem \ref{thm:wrongembeddings} $A$  is finite-dimensional. 
  By  Theorem~\ref{T.Finite} 
 we have $A=M_2(\bbC)$, but 
 $M_2(\bbC)$ has projections of ranks 1 and 2,
and therefore clearly fails the nonunital version of \ref{property:WeakerThanQE}. 
\end{proof} 

In case of nonunital abelian \cstar-algebras, we have the following.




\begin{prop} \label{P.C.C-0} 
The theories of  $\bbC$ and $C_0(2^{\bbN}\setminus \{0\})$ admit quantifier elimination in the  language 
of \cstar-algebras without a symbol for a unit. 
\end{prop} 

\begin{proof} 
Since  $\bbC$ has neither a proper subalgebra nor a nontrivial self-embedding, it has quantifier elimination by  Proposition~\ref{P.QE}. 

Since $C(2^{\bbN})$ is generated by its projections and all countable atomless Boolean algebras are isomorphic,  $C(2^{\bbN})$  is (up to isomorphism) the only separable model of its theory. Let $X=2^{\bbN}\setminus\{0\}$. Since for a locally compact Hausdorff space $Y$  the unitization of $C_0(Y)$ is isomorphic to $C(Y\cup\{\infty\})$ (where $Y\cup\{\infty\}$ is the one-point compactification of $Y$), every separable model of the theory of $C_0(X)$ is isomorphic to $C_0(X)$.  

The above shows that the theory of $C_0(X)$ is $\omega$-categorical, so by Proposition \ref{prop:OmegaCategoricalQE} the proof will be complete if we show that $C_0(X)$ is near ultrahomogeneous.  By Lemma \ref{L.SpectrumType} it suffices to show that if $\bar a,\bar b\in C_0(X)$ are finite tuples of contractions with $j\sigma(\bar a)=j\sigma(\bar b)$, and $\epsilon>0$, there is an automorphism $\Psi$ of $C_0(X)$ such that $\norm{\Psi(a_i)-b_i}<\epsilon$.  This is an immediate consequence of zero-dimensionality of $2^{\bbN}$ and  homogeneity of the algebra
$\Clop(2^{\bbN})$ of its clopen subsets, but we provide  details for the reader's convenience.  

Fix $\bar a,\bar b\in C_0(X)$ with $j\sigma(\bar a)=j\sigma(\bar b)$ and $\epsilon>0$. Since the joint spectrum is defined from the  unitization, $\bar a$ and $\bar b$ have the same joint spectrum as elements of $(C_0(X))^{\dagger}\cong \CCantor$.
Suppose for a moment that $j\sigma(\bar a)$ is finite. With $n$ denoting the cardinality of $j\sigma(\bar a)$ we can find 
 nonzero projections $p_k$, for $k\leq n$, and $\lambda_{ki}\in \bbC$ for $k\leq n$ and $i\leq |\bar a|$ 
  such that $a_i=\sum_{k\leq n} \lambda_{ki} p_k$ for all $i$.  
We arrange that~$0$ (the point removed from $2^{\bbN}$) belongs to the clopen set corresponding to $p_1$ and therefore 
   $\lambda_{1i}=0$ for 
 all $i\leq |\bar a|$ (because $a_i\in C_0(X)$).  
    Since $j\sigma(\bar a)=j\sigma(\bar b)$, there are nonzero projections $q_k$, for $k\leq n$, 
  such that $b_i=\sum_{k\leq n} \lambda_{ki} q_k$ for all $i$. By the choice of $n$, 
  $0$ belongs to the clopen set corresponding to $q_1$.  
  Any automorphism $\Phi_*$ of $\Clop(2^{\bbN})$ sending the clopen set correspond to $p_k$ to the clopen set corresponding to $q_k$ for all $k\leq n$   
  is dual to an automorphism $\Phi$ of $C(2^{\bbN})$ that sends $\bar a$ to $\bar b$. 
We need to ensure that $\Phi$  sends $C_0(X)$ to itself. 
To do so, choose $\Phi_*$ so that, in addition to the above, it  sends the ultrafilter of projections with 
 $0$ in its range to itself. 
  If  $\Phi$ is dual to $\Phi_*$ then its restriction to  $C_0(X)$ is as required. 
  
Now consider the case when $j\sigma(\bar a)$ is not necessarily finite. 
Fix $\e>0$ and let $G\subseteq j\sigma(\bar a)$ be a finite $\e/2$-dense set. 
Since $2^{\bbN}$ is zero-dimensional there exists $\bar a'$ satisfying 
$j\sigma(\bar a')=G$ within distance $\e/2$ of $\bar a$. 
Similarly there exists $\bar b'$  satisfying 
$j\sigma(\bar b')=G$ within distance $\e/2$ of $\bar b$. 
If $\Psi$ is an automorphism of $C_0(X)$ sending $\bar a'$ to $\bar b'$, 
then $\Psi(\bar a)$ is within $\e$ of $\bar b$. This completes the proof. 
 \end{proof}

It seems likely that $\bbC$ and $C_0(2^{\bbN}\setminus \{0\})$ are the only, up to isomorphism\footnote{By homogeneity of the Cantor space, we have that $2^{\bbN}\setminus\{0\}$ and $2^{\bbN}\setminus\{x\}$ are homeomorphic whenever $x\in 2^{\bbN}$, hence for every such $x$ we have $C_0(2^{\bbN}\setminus \{0\})\cong C_0(2^{\bbN}\setminus\{x\})$.}, 
\cstar-algebras that admit quantifier elimination in the language of \cstar-algebras without a symbol for a unit.
By Theorem~\ref{C0} any counterexample would have to be abelian. Some of the results of 
     \cite{EagleGoldVig} may be  relevant.

\section{Model completeness and model companions}\label{sec:Questions}
A theory is said to be \emph{model complete} if every embedding between models of the theory is \emph{elementary}, in the sense of preserving the values of all formulas.
It is easy to see that quantifier elimination implies model completeness, while the converse is false. 
For example, the theory of every finite-dimensional \cstar-algebra is model complete.  
Model completeness is a useful tool in applications of model theory to algebra; for example, the fact that the (discrete) theory of algebraically closed fields is model complete is the key ingredient in a model-theoretic proof of Hilbert's Nullstellensatz (see \cite[Theorem 3.2.11]{Marker}). 

We shall consider two weakenings of model completeness.  A sentence $\sigma$ is $\forall$ (or \emph{universal}) if it is of the form
\[\sup_{\bar x}\phi(\bar x)\]
where $\phi(\bar x)$ is quantifer-free.  Similaly, a sentence is \allexist{} if it is of the form 
\[
\sup_{\bar x}\inf_{\bar y} \phi(\bar x, \bar y)
\]
where $\phi(\bar x, \bar y)$ is quantifier-free.  A theory is \emph{universally axiomatizable} (respectively, \emph{\allexist-axiomatizable}) if it has a set of universal (\allexist{}) axioms.

If a theory $T$ is  model complete then it is preserved by taking inductive limits of its models; indeed, since every embedding between models of such a theory $T$ is elementary, an inductive limit of models of $T$ forms an elementary chain, and each model in the chain is elementarily embedded in the limit model (see \cite[Proposition 7.2]{BenYaacov2008a}), from which it immediately follows that the limit model is again a model of $T$.  
By a  standard preservation theorem  
the set of models of $T$ is closed under taking inductive 
limits if and only if~$T$ is \allexist-axiomatizable (see e.g., \cite[Proposition 2.4.4 (3)]{Muenster}). Therefore  model completeness of a theory implies its \allexist-axiomatizability. 

The other weakening of 
 model completeness that we shall consider is the existence of a model companion. 
 A theory $T^*$ is said to be a \emph{model companion} of theory~$T$ if: (i) every model of $T$ is a submodel of a model of $T^*$ and vice versa, and (ii)~$T^*$ is model complete. 
  For example, the model companion of the theory of fields is the theory of algebraically closed fields.  
The same argument as above shows that 
 if $T$ is \allexist-axiomatizable and $T^*$ is its model companion then $T^*\supseteq T$ (in particular, every model of $T^*$ is also a model of $T$) and that any  theory can have at most one model companion.

 In \cite[Proposition~5.10]{Goldbring2014} it was proved that, assuming Kirchberg's Embedding Problem has a positive  solution, 
the theory of \cstar-algebras does not have a model companion.  Isaac Goldbring observed that our results, together with the methods of \cite{Goldbring2014}, allow us to remove the dependence on Kirchberg's Embedding Problem; we appreciate his allowing us to include a proof here.  
In the proof we will need the following standard model-theoretic fact, a detailed proof of which can be found in \cite[Proposition 2.3.12]{EagleThesis}.

\begin{lem}\label{lem:ModelCompanion}
Let $T$ be a universally axiomatizable theory with model companion~$T^*$.  The following are equivalent:
\begin{enumerate}
\item{
$T^*$ has quantifier elimination,
}
\item{
$T$ has \emph{amalgamation}: Whenever $A, B, C \models T$, and $f : A\to B$ and $g : A\to C$ are embeddings, then there exists $D \models T$ and embeddings $r : B \to D$ and $s : C \to D$ such that $rf=sg$. \qed
}
\end{enumerate}
\end{lem}

For technical reasons related to the process of converting a \cstar-algebra to a multi-sorted structure, the theory $T$ of unital \cstar-algebras is only
 \allexist-axiomatizable (see \cite[p. 485]{Farah2014a}). 
     If we expand the language of unital \cstar-algebras to include predicates for every $^*$-polynomial with complex coefficients in a single variable, then the theory $T_0$ of unital \cstar-algebras in the expanded language \emph{is} universally axiomatizable (again, see \cite[p. 485]{Farah2014a}).  The new theory includes (universal) axioms asserting that the new symbols agree with the $^*$-polynomials they represent. 

\begin{lem} \label{L:T-T-0}
For  $T, T_0, L$, and $L_0$ as above and an $L$-theory $T^*\supseteq T$
  we have the following. 
\begin{enumerate}
\item      The forgetful functor $F$ from the category of models of $T^*\cup T_0$ 
      to the category of models of~$T^*$ is an equivalence of categories. 
\item If $A_0\subseteq B_0$ are models of $T_0$ then 
$A_0$ is an elementary submodel of $B_0$ if and only if $F(A)$ is an elementary submodel of $ F(B)$. 
\item 
 $T^*$ admits elimination of 
quantifiers if and only if $T^*\cup T_0$ does.
\end{enumerate}
\end{lem} 

\begin{proof} Since every $L_0$-term is (provably in $T_0$) equivalent to 
the $L$-term obtained by replacing new function symbols by the corresponding
$L$-terms, every substructure of $A\models T$ has a unique expansion to a model of $T_0$ and  
 (1) follows.  Similarly,  every $L_0$-formula is (provably in $T_0$) 
equivalent to an $L$-formula, and therefore (2) follows. 

 By \cite[Proposition~13.6]{BenYaacov2008a} a
  theory has quantifier elimination if and only if every embedding between substructures of
  its models $M$ and $N$ extends to  an elementary   embedding between 
      $M$ and an elementary extension of $N$. Together with  (1) and (2) 
      this implies (3). 
\end{proof}

 \begin{thm} \label{T:MC} The theory of  unital \cstar-algebras does not have a model companion. 
  \end{thm}
 
 \begin{proof}
 Suppose $T$ has model companion $T^*$. 
 Since $T$ is \allexist-axiomatizable we have  
 $T^*\supseteq T$ and in particular every model of $T^*$ is a unital \cstar-algebra.
  Lemma~\ref{L:T-T-0} (1--2)  implies that  
  $T^*\cup T_0$ is the model companion of  $T_0$.  The amalgamated free product construction for \cstar-algebras (see \cite[II.8.3.5]{BlackadarOpAlg})
shows that $T_0$ has amalgamation. (Note that, to show that $T_0$ has amalgamation, we can consider either the full or the reduced amalgamated free product).  The theory $T_0$ is universally axiomatizable, so by Lemma~\ref{lem:ModelCompanion} $T^* \cup T_0$ has quantifier elimination, and hence Lemma \ref{L:T-T-0}~(3) implies that $T^*$ has quantifier elimination as well.
By Lemma~\ref{L:T-T-0} (1) 
 every \cstar-algebra must embed into a model of~$T^*$, and that is not possible for any of the theories of \cstar-algebra with quantifier elimination listed in Theorem \ref{T0}.
\end{proof}

The Cuntz algebra $\cO_2$ belongs to the important class of strongly self-absorbing \cstar-algebras. 
A \cstar-algebra $D$ is \emph{strongly self-absorbing} (s.s.a.)  if $D\cong D\otimes D$ and the
 embedding of $D$ into $D\otimes D$ that sends $d$ to $d\otimes 1$ is approximately 
 unitarily equivalent to an isomorphism between $D$ and $D\otimes D$ (\cite{Toms2007}).     
 S.s.a.  \cstar-algebras 
 play an important role in the classification program of \cstar-algebras 
and exhibit interesting model-theoretic properties (see \cite[\S 2.2 and \S 4.5]{Fa:Logic} and \cite{SSA}). 

\begin{thm}\label{T.ssa} Assume  $A$ has the same universal theory as an s.s.a. algebra $D$. 
If the theory of $A$ is model complete (or even just \allexist-axiomatizable), then $A$ is elementarily equivalent to $D$. 
\end{thm}

A use of saturation of ultrapowers yields the following standard model-theoretic fact.

\begin{lem}\label{lem:UniversalTheoryUltrapower}
Two metric structures have the same universal theory if and only if each can be embedded in an ultrapower of the other. \qed
\end{lem}

We can now prove Theorem~\ref{T.ssa}.

\begin{proof}[Proof of Theorem~\ref{T.ssa}] 
The proof uses the sandwich argument of \cite[Proposition~3.2]{GoHaSi:Theory}. 

Since $A$ has the same universal theory as $D$, Lemma \ref{lem:UniversalTheoryUltrapower} implies that $D$ embeds into an ultrapower of $A$ and $A$ embeds into an ultrapower of $D$. 
We therefore have a chain
\[
D\to A^{\cU}\to (D^{\cU})^{\cV}
\]
for some ultrafilters $\cU$ and $\cV$.

Since $D$ is s.s.a., every embedding of it into its ultrapower is
 elementary (e.g., \cite[Theorem~2.15]{SSA}). Taking ultrapower of the diagram and iterating the construction, we obtain a sequence of 
embeddings $B_0\to A_0\to B_1 \to A_1\to \dots$ 
such  $B_i\equiv D$,  
 $A_i\equiv A$ and  embeddings $B_i\to B_{i+1}$ are elementary for all~$i$. The inductive limit is elementarily equivalent to $D$ (by the elementarity) and to $A$ (by the 
well-known fact that \allexist-theories are preserved under direct limits, see \cite[Proposition 2.4.4 (3)]{Muenster}), and the conclusion follows. 
\end{proof}  

A purely infinite, simple, separable, and nuclear \cstar-algebra (that is, a \emph{Kirchberg} algebra) 
 is said to be in \emph{standard form} if $A$ is unital and $[1_A]=0$ in $K_0(A)$.
 This is equivalent to $A$ having a unital copy of $\cO_2$ (see e.g., \cite[\S 3]{SSA} or \cite[Proposition 4.2.3]{rordam2002classification}).

\begin{cor} \label{C2} 
If $A$ is a Kirchberg algebra in standard form other than $\cO_2$ then its theory is not 
\allexist-axiomatizable. In particular, 
If $n\geq 2$ then the theory of $M_{n}(\cO_{n+1})$ is not 
\allexist-axiomatizable. 
\end{cor} 

\begin{proof} 
Let $A$ be a Kirchberg algebra in standard form, and suppose that $A$ is \allexist-axiomatizable; we show that $A \cong \cO_2$.  As mentioned above, $A$ contains a unital copy of $\cO_2$.  Conversely, the main result of \cite{KircPhi:Embedding} shows that every separable exact \cstar-algebra embeds in $\cO_2$, so since nuclear algebras are exact we have that $A$ embeds into $\cO_2$.  Lemma \ref{lem:UniversalTheoryUltrapower} therefore implies that $A$ and $\cO_2$ have the same universal theory, and therefore by Theorem \ref{T.ssa} $A$ and $\cO_2$ are elementarily equivalent.  To finish the proof we use the fact that $\cO_2$ is (up to isomorphism) the only separable nuclear model of its theory (this is a consequence of Kirchberg's theorem that $A\otimes \cO_2\cong \cO_2$ for all  separable, nuclear, unital simple \cstar-algebras $A$; see \cite{Goldbring2014} or \cite{Muenster}).  Thus $A \cong \cO_2$. 

The fact that $M_n(\cO_{n+1})$ is in standard form is well-known; see  e.g., the discussion preceding \cite[Proposition 4.2.3]{rordam2002classification} or \cite[Theorem 2.3]{CuntzKTheory}. 
\end{proof}

It is shown in \cite[Proposition 5.7]{Goldbring2014} that a positive solution to the Kirchberg's Embedding Problem implies that $\op{Th}(\mc{O}_2)$ is not model complete. We should also remark that in the case of II$_1$ factors the only strongly self-absorbing algebra is 
the hyperfinite II$_1$ factor $R$ (\cite[Theorem 5.1]{Connes.Class}), and its theory is shown  (relying on \cite{Bro:Topological}) not to be model-complete in \cite{GoHaSi:Theory}. 


Having shown that many natural examples of \cstar-algebras do not have quantifier \emph{elimination}, we may ask whether they have quantifier \emph{reduction}, that is, whether it can be shown that every formula is equivalent to one with a fixed number of alternations of quantifiers.  For example, in the discrete setting Sela \cite{Sela} showed that in the theory of nonabelian free groups every formula is equivalent to a boolean combination of \allexist~formulas.

\begin{q}
Is there a natural example of a \cstar-algebra which admits quantifier reduction?
\end{q}

Given the primarily negative nature of our results, a natural question is to determine if there is a useful expansion of the language (and consequently of the theory) of \cstar-algebras in which wider classes of algebras do have quantifier elimination.  As we described in the introduction, for such an expansion to be useful we should add only a small number of symbols for natural predicates which are definable, but not quantifier-free definable, in the original language for \cstar-algebras. Also, it is necessary in this case to add axioms describing how the interpretation of these symbols should behave, to ensure that the new symbols are interpreted in the intended manner.  Changing the language in this way can change whether or not a class of structures has quantifier elimination, even if it does not change which structures are in the class.  A classical example of this from discrete logic is the theory of real closed fields (see \cite[Section 3.3]{Marker}).  There is a first-order theory in the language of fields whose models are precisely the real closed fields, but this theory does not have quantifier elimination.  Each real closed field admits a unique ordering making it an ordered field, and Tarski showed that the theory of real closed fields in the language of \emph{ordered} fields does have quantifier elimination.

\section{Appendix 
with  Diego Caudillo Amador,   
Bradd Hart,  Jamal Kawach, and Se-jin Kim}

We provide a first step of the proof (completed in  \cite[Corollary 3.4]{EagleGoldVig}) that the only theories of commutative \cstar-algebras that admit elimination 
of quantifiers are $\bbC$, $\bbC^2$ and $C(2^{\bbN})$, where~$2^{\bbN}$ denotes the Cantor space.
 In \cite[Theorem 5.26]{EagleVignati} it was proved that the latter algebra has quantifier elimination.
Since $2^{\bbN}$ is (up to homeomorphism) the unique zero-dimensional, compact metrizable space with no isolated points, 
$C(2^{\bbN})$ is the unique separable model of its theory. Therefore, if $X$ is any compact zero-dimensional 
space with no isolated points then $C(X)$ is elementarily equivalent to~$C(2^{\bbN})$.

\begin{lem} \label{P.A.1} 
Let $C(X)$ be an infinite-dimensional commutative \cstar-algebra that admits elimination of quantifiers.  Then either $X$ is connected, or $C(X)$ is elementarily equivalent to $C(2^\en)$.
\end{lem}

\begin{proof} 
Let $X$ be such that $C(X)$ is infinite-dimensional, $C(X)$ has quantifier elimination, and $X$ is not connected.  By the Downward L\"owenheim-Skolem Theorem (see \cite[Proposition 7.3]{BenYaacov2008a}) we may assume that $C(X)$ is separable, and therefore that $X$ is metrizable.  We begin by observing that $X$ does not have isolated points.  To see this, suppose that $a \in X$ is isolated, and let $p \in C(X)$ be the characteristic function of $\{a\}$.  Then $p$ is a scalar projection (as defined in Section \ref{section:findim}).  Since $C(X)$ has quantifier elimination it also satisfies property \ref{property:WeakerThanQE}, so $C(X)$ is an infinite-dimensional algebra with \ref{property:WeakerThanQE} and a scalar projection, contradicting Proposition \ref{P.abelian}.

Let $A \subseteq X$ be a nontrivial clopen set, and let $p$ be the characteristic function of $A$.  Then $p$ is a nontrivial projection in $C(X)$.  Since $X$ has no isolated points and $A$ is open in $X$ it follows that $A$ also has no isolated points, and hence that $C(A) \cong pC(X)p$ has no scalar projections.  Therefore by Lemma \ref{L.MASA} there is $f\in pC(X)p$ with $\sigma(f)=[0,1]$.  The same argument applied to $X \setminus A$ gives $g\in (1-p)C(X)(1-p)$ with $\sigma(g)=[0,1]$. Let $h=\frac{f+p+g}{2}$.  We have that $\norm{h}=1$ and $\sigma(h)=[0,1]$.  Consider the following formula:
\[\psi(x)=\inf_{q=q^*=q^2}\max\{\norm{qx}, \norm{(1-q)(1-x)}\}\]
The above expression is a bona fide formula since the quantification appearing in it is over a definable set - see \cite[Theorem~3.2.2]{Muenster}.
Our choice of $h$ implies that $\psi^{C(X)}(h) = \frac{1}{2}$.

The Cantor space is the unique compact totally disconnected metrizable space without isolated points, so to complete the proof it suffices to show that $X$ is totally disconnected.  Suppose to the contrary that $Y \subseteq X$ is a closed connected set with at least two distinct points $y, z \in Y$.  Let $f \in C(X)$ be such that $\sigma(f) = [0, 1]$, $f(y) = 0$, and $f(z) = 1$.  By Lemma \ref{L.SpectrumType} the functions $f$ and $h$ have the same quantifier-free type, and so since $C(X)$ has quantifier elimination $f$ and $h$ have the same type.  In particular, $\psi^{C(X)}(f) = \frac{1}{2}$.  We can therefore find a projection $q \in C(X)$ such that $\norm{qf}<\frac{2}{3}$ and $\norm{(1-q)(1-f)}<\frac{2}{3}$. Let $B=q^{-1}\left(\{1\}\right)$. If $z\in B$, then $q(z)f(z)=1$, contradicting $\norm{qf}<\frac{2}{3}$. Conversely if $y\notin B$ then $(1-q)(y)(1-f)(y)=1$, a contradiction to $\norm{(1-q)(1-f)}<\frac{2}{3}$. Thus $B\cap Y$ and $(X\setminus B)\cap Y$ disconnect $Y$.
\end{proof}

Recall that a \cstar-algebra is said to have \emph{real rank zero} if every self-adjoint element can be approximated by self-adjoint elements of finite spectrum, and that for $X$ a compact metrizable space the real rank of $C(X)$ coincides with the Lebesgue covering dimension of $X$ (see \cite{BrownPed91}).  It follows from Lemma \ref{P.A.1} and Theorem \ref{T.Finite} that the only theories of unital commutative real rank zero \cstar-algebras with quantifier elimination are the theories of $\mbb{C}, \mbb{C}^2$, and $C(2^\en)$.  

We now turn to the other side of the dichotomy in Lemma \ref{P.A.1}, and consider the case where $X$ is connected.  Recall that a connected compact Hausdorff space (i.e., a \emph{continuum}) is said to be \emph{indecomposable} if it is not the union of two of its proper subcontinua. This property is equivalent (see e.g., \cite[\S 48, Theorem 2]{Kurat}) to every connected open subset of $X$ being dense.

\begin{thm} \label{T.A.1} 
If $X$ is a continuum such that $C(X)$ has elimination of quantifiers then $X$ is indecomposable.
\end{thm}

\begin{proof}We work by contradiction, and assume that $X$ is not indecomposable. For the purpose of the proof, we say that a function $f \in C(X)$ is a \emph{peak function} if $\sigma(f)=[0,1]$ and the set $\{x \in X: f(x)>4/5\}$ is connected.  Also for this proof we say that $f \in C(X)$ is a \emph{volcano function} if $\sigma(f)=[0,1]$ and $f=g+h$ for some $g$ and $h$ that satisfy $\sigma(g)=\sigma(h)=[0,1]$ and $gh=0$.  By using continuous functional calculus and Lemma~\ref{L.MASA} we see that every commutative \cstar-algebra with no scalar projections contains a volcano function, so let $f_1 \in C(X)$ be a volcano function. We shall construct a peak function $f_2$ and show that $f_1$ and $f_2$ have different types.  The desired contradiction is then obtained because Lemma~\ref{L.SpectrumType} implies that $f_1$ and $f_2$ have the same quantifier-free type.

Let $U$ be a connected open subset which is not dense in $X$, and fix $z\in X$ such that $\dist(z,U)>r>0$ for some $r$.  
With $F=X\setminus U$, the function 
$
h_0(x)=\dist(x,F)
$
is nonzero only on $U$. We normalize and let  $h=\|h_0\|^{-1} h_0$. 
The function $g(x)=r^{-1}\max(0,r-d(x,z))$ satisfies $\sigma(g)=[0,1]$ and $g$ is identically $0$ on $U$.
Let $f_2 = \frac{1}{5}h + \frac{4}{5}(1-g)$.  We claim that $f_2$ is a peak function.  We clearly have $\sigma(f_2) \subseteq [0, 1]$, and we also have $f_2(z) = 0$ and $f_2(x) = 1$ whenever $h(x) = 1$, so since $X$ is connected $\sigma(f_2) = [0, 1]$.  If $x \in U$ then $g(x) = 0$, so $f_2(x) > \frac{4}{5}$, while if $x \not\in U$ then $h(x) = 0$ so $f_2(x) \leq \frac{4}{5}$.  Therefore $\{x \in X : f_2(x) > \frac{4}{5}\} = U$, which is connected.

We now show  that $f_1$ and $f_2$ do not have the same type. 
Consider the formula
\[
\phi(f)=\inf_{c,d} \max(\|f-cc^*-dd^*\|, |1-\|c\||, |1-\|d\||, \|cc^*dd^*\|).
\]
Writing $f_1 = g_1 + h_1$ as in the definition of being a volcano function, and taking $c=g_1^{1/2}$ and $d=h_1^{1/2}$ we see that  $\phi^{C(X)}(f_1)=0$. 

Assume   $\phi^{C(X)}(f_2)<1/10$. Then there are  
$a=cc^*$ and $b=dd^*$ such that 
\[
\textstyle \max(\|f_2-a-b\|, |1-\|a\||, |1-\|b\||, \|ab\|)< \frac{1}{10}. 
\]
In particular there are $s,t\in X$ such that $a(s)>\frac{9}{10}$ and $b(t)>\frac{9}{10}$. 
Since $|f_2(x)-a(x)-b(x)|<\frac{1}{10}$ and $a(x),b(x)$ are positive for all $x\in X$, 
we have \[f_2(s)>\frac{4}{5}, \,\,f_2(t)>\frac{4}{5},\,\,a(t)<\frac{1}{5}\text{ and }b(s)<\frac{1}{5}.\] 
 Let $Z_1=\{x\in X\colon a(x)\leq b(x)\}$ and $Z_2=\{x\in X\colon b(x)\leq a(x)\}$. Then $U=\{x\in X\colon f_2(x)>\frac{4}{5}\}$ can be covered by $Z_1\cap U$ and $Z_2\cap U$. Since $U$ is connected, there is $x\in U\cap Z_1\cap Z_2$. For such an $x$ we have $a(x)=b(x)$.
 By calculation, from the fact that $x\in U$, $\|f_2-a-b\|<\frac{1}{10}$, and $\|ab\|<\frac{1}{10}$, we get that $a(x)=b(x)>\frac{7}{20}$ and $a(x)b(x)>\frac{1}{10}$. This violates our assumptions and completes the proof. 
\end{proof} 

\bibliographystyle{amsalpha}
\bibliography{qe-cstar}
\end{document}